\documentclass{amsart}
 \newtheorem{theorem}{Theorem}[section]
 \newtheorem{corollary}[theorem]{Corollary}
 \newtheorem{lemma}[theorem]{Lemma}
 \newtheorem{proposition}[theorem]{Proposition}
 \theoremstyle{definition}
 
 \theoremstyle{remark}
 \newtheorem{remark}[theorem]{Remark}
 \newtheorem{remarks}[theorem]{Remarks}
 
 \numberwithin{equation}{section}

\RequirePackage{amssymb}
\newcommand{\highlight}[1]{{\it #1}}
\newcommand{\rank}{\mathop{\mathrm{rk}}\nolimits}
\newcommand{\supp}{\mathop{\mathrm{supp}}\nolimits}
\newcommand{\height}{\mathop{\mathrm{ht}}\nolimits}
\newcommand{\soc}[1]{\mathcal{I}_{#1}}
\newcommand{\group}[1]{\mathbf{#1}}

\begin{document}

\title[On adjoint orbits in nilpotent ideals of a Borel subalgebra]{On adjoint orbits in nilpotent ideals of a Borel subalgebra}

\author[Rupert Yu]{Rupert W. T. Yu}

\address{%
Laboratoire de Math\'ematiques de Reims UMR 9008 CNRS\\ 
Universit\'e de Reims Champagne Ardenne\\
Moulin de la Housse - BP 1039\\
51687 REIMS cedex 2\\
France}

\email{rupert.yu@univ-reims.fr}

\begin{abstract}
Let $\mathfrak{m}$ be a nilpotent ideal in the Borel subalgebra $\mathfrak{b}$ of a complex 
finite-dimensional semisimple Lie algebra, and $\mathfrak{m}^{\bullet}$ the subset of 
(ad-)nilpotent elements in $\mathfrak{b}$ such that $\mathfrak{m}$ is the minimal ideal containing them. 
This set is stable under the adjoint action of the corresponding Borel subgroup $\group{B}$. We prove that 
$\mathfrak{m}^{\bullet}$ contains a unique closed $\group{B}$-orbit which is the orbit of a nilpotent 
element whose support is the set of minimal roots associated to the root space decomposition 
of $\mathfrak{m}$.
\end{abstract}



\maketitle

\section{Introduction}

Let $\mathfrak{b}$ be a Borel subalgebra of a complex finite-dimensional semisimple Lie algebra 
$\mathfrak{g}$, $\group{G}$ the adjoint algebraic group of $\mathfrak{g}$ and $\group{B}$ the 
Borel subgroup of $\group{G}$ whose Lie algebra is $\mathfrak{b}$.
We fix a decomposition $\mathfrak{b} = \mathfrak{h}\oplus \mathfrak{n}$ where $\mathfrak{n}$
is the nilradical of $\mathfrak{b}$ and $\mathfrak{h}$ is a Cartan subalgebra of $\mathfrak{b}$.
Observe that $\mathfrak{n}$ is also the set of (ad-)nilpotent elements in $\mathfrak{b}$.

Adjoint nilpotent $\group{G}$-orbits in $\mathfrak{g}$ are classified by many combinatorial objects. 
They are well studied,
and have many nice characterisations and properties (finiteness, algebraic and geometric). In contrast, 
$\group{B}$-orbits of $\mathfrak{n}$ are less well understood. There are in general an infinite number
of them, and they are not necessary conical. Kashin \cite{K} determined exactly when there is a 
finite number of $\group{B}$-orbits in $\mathfrak{n}$. Hille and R\"ohrle \cite{HR} obtained a more
general result using a quiver model. Boos \cite{B} and Melnikov \cite{M1} studied $\group{B}$-orbits 
in certain subvarieties
of $\mathfrak{n}$ in type $A$. Note also that
connected components of the intersection of a nilpotent $\group{G}$-orbit with $\mathfrak{b}$ are called 
orbital varieties which are related to irreducible components of the corresponding Springer 
fibers \cite{S,FM}, and relations between nilpotent $\group{G}$-orbits and nilpotent ideals in 
$\mathfrak{b}$ are studied in \cite{So2}.

Since $\mathfrak{g}$ is semisimple,
an ideal in $\mathfrak{b}$ is nilpotent if and only if it is contained in $\mathfrak{n}$.
For any nilpotent ideal $\mathfrak{m}$ of $\mathfrak{b}$, we denote by $\mathfrak{m}^{\bullet}$
the set of $X\in \mathfrak{n}$ such that $\mathfrak{m}$ is the minimal nilpotent ideal in $\mathfrak{b}$
containing $X$. Thus $\mathfrak{n}$ is the disjoint union of $\mathfrak{m}^{\bullet}$ where $\mathfrak{m}$
runs through all nilpotent ideals of $\mathfrak{b}$.

Let $\Delta \supset \Delta^{+} \supset \Pi$ be respectively the root system of $\mathfrak{g}$, the 
set of positive roots and the set of simple roots relative to $( \mathfrak{b} , \mathfrak{h})$. Denote by 
$\mathfrak{g}_{\alpha}$ the root subspace relative to $\alpha$ in $\mathfrak{g}$. 
For any nilpotent ideal $\mathfrak{m}$ of $\mathfrak{b}$, there exists a unique subset 
$\Delta_{\mathfrak{m}}^{+}$ of $\Delta^{+}$ such that
$$
\mathfrak{m} = \bigoplus_{\alpha\in \Delta^{+}_{\mathfrak{m}}} \mathfrak{g}_{\alpha}.
$$
Denote by $\soc{\mathfrak{m}}$ the set of minimal roots in $\Delta_{\mathfrak{m}}^{+}$
with respect to the usual partial order $\prec$ on $\Delta$ given by 
$\alpha \prec \beta$ if $\beta - \alpha$ is a sum of simple roots. 
The main result of the paper is the following theorem.

\begin{theorem}\label{main-theorem}
Let $\mathfrak{m}$ be a nilpotent ideal of $\mathfrak{b}$ and 
$X\in \mathfrak{m}^{\bullet}\cap \bigoplus_{\alpha \in \soc{\mathfrak{m}}} \mathfrak{g}_{\alpha}$.
Then the $\group{B}$-orbit of $X$ is the unique closed $\group{B}$-orbit in the $\group{B}$-stable 
subvariety $\mathfrak{m}^{\bullet}$.
\end{theorem}

The key observation here is that being a set of pairwise non comparable positive
roots, a result of Sommers \cite{So} says that $\soc{\mathfrak{m}}$ is conjugate 
under the Weyl group of $\mathfrak{g}$ to a subset of $\Pi$. This allows us 
in section \ref{sec:uniqueness} to obtain sufficient information on $\mathfrak{n}$-invariant functions on 
$\mathfrak{m}$ in order to prove that $\mathfrak{m}^{\bullet}$ contains a unique 
closed $\group{B}$-orbit. In section \ref{sec:description}, we prove a nice geometrical
property (Theorem \ref{antichain-hyperplane})
for sets consisting of pairwise non comparable positive roots which allows us to precise 
exactly this closed orbit by studying orbits under the Cartan subgroup $\group{T}$ in 
$\group{B}$ whose Lie algebra is $\mathfrak{h}$.
\smallskip

\paragraph{Notations.} We shall conserve the notations above in the rest of the paper. Denote by 
$\group{U}$ the unipotent radical of $\group{B}$ which is the subgroup of $\group{B}$ generated by 
$\exp (\mathrm{ad} (X) )$, 
$X\in \mathfrak{n}$. Recall that $\mathfrak{n}$ is the Lie algebra of $\group{U}$ and 
$\group{B} = \group{T} \ltimes \group{U}$.

For any linearly independent subset $\soc{}$ of $\Delta$ and 
$\mathbb{A} \subset \mathbb{C}$, we shall denote by 
$\mathbb{A} \soc{}$ the set of elements of $\mathfrak{h}^{*}$ of the form
$\sum_{\alpha\in \soc{}} c_{\alpha} \alpha$ where $c_{\alpha} \in \mathbb{A}$ for all $\alpha\in\soc{}$.
Given $\beta =  \sum_{\alpha\in \soc{}} c_{\alpha} \alpha \in \mathbb{A} \soc{}$, we shall denote by 
$\height_{\soc{}} (\beta ) =\sum_{\alpha\in \soc{}} c_{\alpha}$ the height of $\beta$ with respect to 
$\soc{}$.

\section{Nilpotent elements, nilpotent ideals and antichains}\label{sec:socle}

Let us fix a basis $(X_{\alpha})_{\alpha\in\Delta^{+}}$ of $\mathfrak{n}$ where 
$X_{\alpha}\in \mathfrak{g}_{\alpha}$.
For $X = \sum_{\alpha\in \Delta^{+}} c_{\alpha} X_{\alpha} \in\mathfrak{n}$, we set 
$$
\supp (X) = \{ \alpha ; c_{\alpha} \neq 0\} \ 
$$
the support of $X$, $\soc{X}$ the set of minimal roots in $\supp(X)$ and $\mathfrak{m}_{X}$ 
the minimal nilpotent ideal in $\mathfrak{b}$ containing $X$.
The following lemmas are immediate from the definitions.

\begin{lemma}\label{socle-element}
Let $X\in \mathfrak{n}$ and $\Delta_{X}^{+} 
= \{ \alpha \in \Delta^{+} ; \alpha \succcurlyeq \beta$ for some 
$\beta\in \soc{X} \}$.
\begin{enumerate}
\item We have $\mathfrak{m}_{X} = \bigoplus_{\alpha\in \Delta_{X}^{+}} \mathfrak{g}_{\alpha}$.
In particular, $\soc{X} = \soc{\mathfrak{m}_{X}}$.
\item For any $\sigma \in \group{B}$, we have $\mathfrak{m}_{X} = \mathfrak{m}_{\sigma ( X )}$ and 
$\soc{X}=\soc{\sigma( X )}$.
\end{enumerate}
\end{lemma}

\begin{lemma}\label{socle-ideal}
Let $\mathfrak{m}$ be a nilpotent ideal of $\mathfrak{b}$ and
$\mathfrak{m}^{\circ}$ be the union of all nilpotent ideals in $\mathfrak{b}$ 
contained strictly in $\mathfrak{m}$. Then
$$
\mathfrak{m}^{\bullet} = \{ X\in \mathfrak{n}\, ; \Delta_{\mathfrak{m}}^{+}\supset 
\supp( X) \supset \soc{\mathfrak{m}} \}
= \mathfrak{m} \setminus  \mathfrak{m}^{\circ}.
$$
is an open $\group{B}$-stable subset of $\mathfrak{m}$.
\end{lemma}

Let $\mathfrak{m}$ be a nilpotent ideal of $\mathfrak{b}$. Then it is completely determined
by $\soc{\mathfrak{m}}$. Being minimal elements of $\Delta_{\mathfrak{m}}^{+}$, 
the set $\soc{\mathfrak{m}}$ contains pairwise non comparable roots, also called an
\highlight{antichain}. Conversely, any antichain $\soc{}$ in $\Delta^{+}$ defines 
a nilpotent ideal $\mathfrak{m}_{X}$ where $X$ is any element in $\mathfrak{n}$ 
verifying $\supp (X) = \soc{}$. Thus the map $\mathfrak{m}\mapsto \soc{\mathfrak{m}}$ 
is a bijection between the set of nilpotent ideals in $\mathfrak{b}$ and the set of antichains
in $\Delta^{+}$.  

\begin{theorem}[Sommers \cite{So}]\label{antichain}
Let $\soc{}$ be an antichain in $\Delta^{+}$. There exists an element $w$ in the 
Weyl group of $\mathfrak{g}$ such that $w (\soc{} ) \subset \Pi$.
\end{theorem}

\begin{remark}\label{antichain-remark}
We obtain immediately from Theorem \ref{antichain} that if
$\soc{}$ is a non empty antichain in $\Delta^{+}$, then the set $\Delta_{\soc{}} = \mathbb{Z} \soc{} \cap \Delta =
\mathbb{Q} \soc{} \cap \Delta$ is a reduced root system, $\soc{}$ is a set of simple roots of $\Delta_{\soc{}}$
and $\Delta_{\soc{}}^{+} = \Delta_{\soc{}} \cap \Delta = \mathbb{N} \soc{} \cap \Delta$
is the set of positive roots with respect to $\soc{}$.
\end{remark}

\section{Invariant functions and closed $\group{B}$-stable subsets}\label{sec:uniqueness}

Let $\mathfrak{m}$ be a nilpotent ideal in $\mathfrak{b}$. 
By Remark \ref{antichain-remark}, $\Delta_{\soc{\mathfrak{m}}}$ is a reduced root system. Therefore
$
\mathfrak{g}_{0} = \mathfrak{h} \oplus \bigoplus_{\alpha\in \Delta_{\soc{\mathfrak{m}}}}
\mathfrak{g}_{\alpha}
$
is a reductive Lie subalgebra in $\mathfrak{g}$. Let 
$$
\mathfrak{m}_{0} = \bigoplus_{\alpha\in \Delta_{\soc{\mathfrak{m}}}^{+}}
\mathfrak{g}_{\alpha} \, \hbox{ and } \,
\mathfrak{b}_{0} = \mathfrak{h}\oplus \mathfrak{m}_{0}
$$
which is a Borel subalgebra of $\mathfrak{g}_{0}$ whose nilradical is 
$\mathfrak{m}_{0}$.
Denote by $\group{B}_{0}$ the closed connected subgroup of $\group{G}$ 
whose Lie algebra is $\mathfrak{b}_{0}$.

Let $(\xi_{\alpha})_{\alpha\in \Delta_{{\mathfrak{m}}}^{+}}$ be the dual basis of 
the basis $(X_{\alpha})_{\alpha\in \Delta_{{\mathfrak{m}}}^{+}}$ of $\mathfrak{m}$.
The ring of regular functions on $\mathfrak{m}$ decomposes
as a direct sum 
$$
\mathbb{C} [\mathfrak{m}] = \mathbb{C} [\xi_{\alpha}, \alpha\in \Delta_{\mathfrak{m}}^{+}] = 
\bigoplus_{\alpha\in -\mathbb{N}\Pi} \mathbb{C} [\mathfrak{m}]_{\alpha}
$$
where 
$\mathbb{C} [\mathfrak{m}]_{\alpha}$ denotes
the vector space span of monomials  $\xi_{\alpha_{1}}\cdots \xi_{\alpha_{r}}$ where
$r\in\mathbb{N}$, $\alpha_{1},\dots ,\alpha_{r}\in\Delta_{\mathfrak{m}}^{+}$ and
$\alpha_{1}+\cdots + \alpha_{r} = -\alpha$. Since $H ( \xi_{\alpha} ) = -\alpha (H) \xi_{\alpha}$
for any $H\in\mathfrak{h}$, 
$\mathbb{C} [\mathfrak{m}]_{\alpha}$ is also the $\mathfrak{h}$-weight 
subspace of $\mathbb{C} [\mathfrak{m}]$ of weight $\alpha$.

Note also that for $\alpha,\beta \in \Delta_{\mathfrak{m}}^{+}$, 
$X_{\alpha} ( \xi_{\beta} )$ is a non zero multiple of $\xi_{\beta-\alpha}$ if  
$\beta - \alpha \in \Delta_{\mathfrak{m}}^{+}$, and is zero otherwise.

\begin{proposition}\label{invariant-prop}
We have
$$
\bigoplus_{\alpha\in -\mathbb{N}\soc{\mathfrak{m}}} \mathbb{C} [\mathfrak{m}]_{\alpha}^{\mathfrak{n}}
=
\bigoplus_{\alpha\in -\mathbb{N}\soc{\mathfrak{m}}} \mathbb{C} [\mathfrak{m}]_{\alpha}^{\mathfrak{m}_{0}}
= \mathbb{C}[ \xi_{\alpha} , \alpha \in \soc{\mathfrak{m}}].
$$
\end{proposition}
\begin{proof}
Let 
$$
\mathfrak{c} = \bigoplus_{\alpha\in\Delta_{\mathfrak{m}}^{+} \setminus \Delta_{\soc{\mathfrak{m}}}^{+}}
\mathfrak{g}_{\alpha}.
$$
Thus we have $\mathfrak{m} = \mathfrak{m}_{0} \oplus \mathfrak{c}$ as $\mathfrak{b}_{0}$-modules.
As the composition of the natural morphisms of $\mathfrak{b}_{0}$-modules
$$
\mathfrak{m}_{0}\stackrel{\iota}{\longrightarrow} \mathfrak{m} =\mathfrak{m}_{0} \oplus \mathfrak{c} 
\stackrel{\pi}{\longrightarrow} \mathfrak{m}_{0}
$$
is the identity map, the composition of the comorphisms
$$
\mathbb{C}[ \mathfrak{m}_{0}] \stackrel{\pi^{*}}{\longrightarrow} 
\mathbb{C}[ \mathfrak{m}] \stackrel{\iota^{*}}{\longrightarrow} \mathbb{C}[\mathfrak{m}_{0}]
$$
is also the identity map. We shall identify $\mathbb{C}[\mathfrak{m}_{0}]$ with its image under $\pi^{*}$ 
which is precisely $\mathbb{C}[ \xi_{\alpha} , \alpha \in \Delta_{\soc{\mathfrak{m}}}^{+}]$. Hence 
$$
\mathbb{C}[ \mathfrak{m}]  = \mathbb{C}[ \mathfrak{m}_{0}] \oplus \ker (\iota^{*})
$$
as $\mathfrak{b}_{0}$-modules. Consequently, $\mathbb{C}[ \mathfrak{m}]^{\mathfrak{m}_{0}}
= \mathbb{C}[ \mathfrak{m}_{0}]^{\mathfrak{m}_{0}} \oplus \ker (\iota^{*})^{\mathfrak{m}_{0}}$.

Suppose that $\ker (\iota^{*})^{\mathfrak{m}_{0}}_{\lambda} \neq 0$ for some 
$\lambda\in -\mathbb{N}\soc{\mathfrak{m}}$. Fix such a 
$\lambda = \sum_{\alpha\in\Pi} c_{\alpha}\alpha$ such that $\height_{\Pi} (\lambda )$
is maximal. Let 
$\varphi \in \ker (\iota^{*})^{\mathfrak{m}_{0}}_{\lambda}$ be 
non zero. There exist $\beta\in \Delta_{\mathfrak{m}}^{+}$ and $n > 0$ such that
$$
\varphi = \psi_{n} \xi_{\beta}^{n} + \psi_{n-1} \xi_{\beta}^{n-1} + \cdots  + \psi_{1} \xi_{\beta} + \psi_{0}
$$
where $\psi_{k} \in \mathbb{C}[ \xi_{\alpha} , \alpha \not\succcurlyeq \beta]_{\lambda + k \beta}$
and $\psi_{n}\neq 0$. 

For any $Z \in \mathfrak{m}_{0}$, we have
$$
0 = Z ( \varphi )  = Z ( \psi_{n}) \xi_{\beta}^{n} + \eta
$$
where by our choice of $\beta$, the degree of $\eta$ in $\xi_{\beta}$ is at most $n-1$. 
So $Z (\psi_{n} ) = 0$.
Thus $\psi_{n} \in \mathbb{C}[\mathfrak{m}]^{\mathfrak{m}_{0}}_{\lambda + n\beta}$, and our choice
of $\lambda$ implies that $\psi_{n} \in \mathbb{C}[\mathfrak{m}_{0}]^{\mathfrak{m}_{0}}$, and hence
$\lambda + n\beta \in -\mathbb{N}\soc{\mathfrak{m}}$. 
It follows that 
$\beta \in \mathbb{Q} \soc{\mathfrak{m}}\cap \Delta_{\mathfrak{m}}^{+} = 
\mathbb{N}\soc{\mathfrak{m}}\cap \Delta_{\mathfrak{m}}^{+} = \Delta_{\soc{\mathfrak{m}}}^{+}$
(Remark \ref{antichain-remark}), and 
so $\psi_{n}\xi_{\beta}^{n}\in \mathbb{C}[\mathfrak{m}_{0}]$
which is impossible since $\varphi \in \ker (\iota^{*})^{\mathfrak{m}_{0}}_{\lambda}$. 
We have therefore established that
$\bigoplus_{\alpha\in -\mathbb{N}\soc{\mathfrak{m}}} 
\mathbb{C} [\mathfrak{m}]_{\alpha}^{\mathfrak{m}_{0}} = 
\mathbb{C}[\mathfrak{m}_{0}]^{\mathfrak{m}_{0}}$. 

Now, we check readily that 
$\mathbb{C}[\xi_{\alpha} , \alpha \in \soc{\mathfrak{m}}]$ is contained in both 
$\mathbb{C}[\mathfrak{m}_{0}]^{\mathfrak{m}_{0}}$ and 
$\mathbb{C}[\mathfrak{m}]^{\mathfrak{n}}$.
We claim that 
$\mathbb{C}[\mathfrak{m}_{0}]^{\mathfrak{m}_{0}} = \mathbb{C}[\xi_{\alpha} , \alpha \in \soc{\mathfrak{m}}]$.

Let us prove our claim. Observe that the ring of $\mathfrak{m}_{0}$-invariant regular functions on 
$\mathfrak{m}_{0}$ is the ring of $\mathfrak{b}_{0}$-semi-invariant regular functions on $\mathfrak{m}_{0}$.
By Richardson's dense orbit theorem \cite{R}, $\mathfrak{m}_{0}$ contains an open $\group{B}_{0}$-orbit. 
It follows from a theorem of Rosenlicht \cite{Ro} that the field of $\group{B}_{0}$-invariant rational functions 
is $\mathbb{C}$. This implies that $\dim \mathbb{C}[\mathfrak{m}_{0}]^{\mathfrak{m}_{0}}_{\alpha} \leqslant 1$
for any $\alpha \in -\mathbb{N}\soc{\mathfrak{m}}$. Since  
$\mathbb{C}[\xi_{\alpha} , \alpha \in \soc{\mathfrak{m}}] \subset 
\mathbb{C}[\mathfrak{m}_{0}]^{\mathfrak{m}_{0}}$,
we have our claim.

We deduce immediately from our claim the required equalities since 
$\mathbb{C}[\mathfrak{m}_{0}]^{\mathfrak{m}_{0}} = 
\mathbb{C}[\xi_{\alpha} , \alpha \in \soc{\mathfrak{m}}] 
\subset \mathbb{C}[\mathfrak{m}]^{\mathfrak{n}} \subset
\mathbb{C} [\mathfrak{m}]^{\mathfrak{m}_{0}}$. 
\end{proof}

\begin{theorem}\label{intersection-theorem}
The intersection of two non empty closed $\group{B}$-stable subsets of $\mathfrak{m}^{\bullet}$
is non empty.
\end{theorem}
\begin{proof}
Let us conserve the notations in the proof of Proposition \ref{invariant-prop}.
Set $\group{M}_{0}$  to be the closed connected subgroup
of $\group{B}_{0}$ whose Lie algebra is $\mathfrak{m}_{0}$.
In particular, $\group{M}_{0}$ is the subgroup generated by $\exp (\mathrm{ad} (X))$
where $X\in \mathfrak{m}_{0}$.

By Proposition \ref{invariant-prop} and Lemma \ref{socle-ideal}, $\varphi = \prod_{\alpha \in \soc{\mathfrak{m}}} \xi_{\alpha} \in 
\mathbb{C}[\mathfrak{m}_{0}]^{\mathfrak{m}_{0}}$
and $\mathfrak{m}^{\bullet} = \{ X \in \mathfrak{m} ; \varphi (X) \neq 0\}$. 
So
$\mathbb{C}[\mathfrak{m}^{\bullet}]$ is the localization
of $\mathbb{C}[\mathfrak{m}]$ by the multiplicative subset generated by $\varphi$.

Let $\mathfrak{m}_{0}^{\bullet} = \mathfrak{m}_{0} \cap \mathfrak{m}^{\bullet}$. Similarly, 
$\mathbb{C}[\mathfrak{m}_{0}^{\bullet} ]$  is the localization 
of $\mathbb{C}[\mathfrak{m}_{0}]$ by the multiplicative subset generated by $\varphi$.
As in the proof of Proposition \ref{invariant-prop}, we obtain from the composition 
$
\mathfrak{m}_{0}^{\bullet} \stackrel{\iota}{\rightarrow} \mathfrak{m}^{\bullet} 
\stackrel{\pi}{\rightarrow} \mathfrak{m}_{0}^{\bullet}
$
of $\group{B}_{0}$-equivariant morphisms that 
$
\mathbb{C}[\mathfrak{m}^{\bullet} ]= \mathbb{C}[\mathfrak{m}_{0}^{\bullet} ]
\oplus \ker (\iota^{*})
$
as $\group{B}_{0}$-modules. 

Let $J$ be a proper $\group{B}$-stable ideal in $\mathbb{C}[\mathfrak{m}^{\bullet} ]$.
Then $J + \ker (\iota^{*})$ is $\group{B}_{0}$-stable, and hence
$
(J + \ker (\iota^{*}) ) /  \ker (\iota^{*}) \simeq \iota^{*} ( J )
$
is a $\group{B}_{0}$-stable ideal in $\mathbb{C}[\mathfrak{m}_{0}^{\bullet} ]$.

As mentioned in the proof of Proposition \ref{invariant-prop},
$\mathfrak{m}_{0}$ contains an open $\group{B}_{0}$-orbit. Since $\mathfrak{m}_{0}^{\bullet}$
is open in $\mathfrak{m}_{0}$ (Lemma \ref{socle-ideal}), we deduce that either $\iota^{*} ( J ) = 0$ or 
$\iota^{*} ( J ) = \mathbb{C}[\mathfrak{m}_{0}^{\bullet} ]$.

Suppose that $\iota^{*} ( J ) = \mathbb{C}[\mathfrak{m}_{0}^{\bullet} ]$.
Then there exists $(f,g) \in J_{0} \times \ker (\iota^{*})_{0}$ such that 
$f+g =1$. In particular, $f$ and $g$ are non constant since $J$ is proper.

Let $V$ be the $\group{M}_{0}$-submodule generated by $f$ in $\mathbb{C}[\mathfrak{m}^{\bullet} ]$. 
Then $V$ is finite-dimensional, and we have $V^{\group{M}_{0}}\neq 0$ because $\group{M}_{0}$ is unipotent. 
Let $\sigma_{1},\dots ,\sigma_{r}\in \group{M}_{0}$ and $c_{1},\dots ,c_{r}\in \mathbb{C}^{*}$
be such that 
$\tilde{f} = \sum_{i=1}^{r} c_{i}\sigma_{i}( f )$ is a non zero element of $V^{\group{M}_{0}}= V^{\mathfrak{m}_{0}}$.

Set $\tilde{g} = \sum_{i=1}^{r} c_{i}\sigma_{i}( g)$. Then $\tilde{f} + \tilde{g} = \sum_{i=1}^{r} c_{i}$.
As $J$ and $\ker (\iota^{*})$ are $\group{B}_{0}$-stable, we deduce that $\tilde{f} \in J^{\mathfrak{m}_{0}}$, and 
hence $\tilde{g}\in \ker (\iota^{*})^{\mathfrak{m}_{0}}$. Moreover,  $\tilde{f}$ and $\tilde{g}$ are both non constant
because $J$ is proper, and by construction, they are sums of vectors of weights
in $-\mathbb{N}\soc{\mathfrak{m}}$ because $f$ and $g$ are of weight $0$. 

Since $\tilde{g}\in \ker (\iota^{*})^{\mathfrak{m}_{0}}$ and 
$\varphi \in \mathbb{C}[\mathfrak{m}]^{\mathfrak{m}_{0}}$
(Proposition \ref{invariant-prop}), there exists $n>0$ such that $\varphi^{n}\tilde{g}\in 
\bigoplus_{\alpha\in -\mathbb{N}\soc{\mathfrak{m}}} \mathbb{C}[\mathfrak{m}]_{\alpha}^{\mathfrak{m}_{0}}
=\mathbb{C}[\xi_{\alpha} , \alpha\in \soc{\mathfrak{m}}]$ (Proposition \ref{invariant-prop}). 
This implies that $\tilde{g}\in \mathbb{C}[\mathfrak{m}_{0}^{\bullet}]$ which is absurd because $\tilde{g}\in
\ker (\iota^{*})$ and $\tilde{g}\neq 0$. 

We conclude that $\iota^{*}( J) = 0$, and hence $J \subset \ker (\iota^{*}) \neq \mathbb{C}[\mathfrak{m}^{\bullet}]$.

Consequently, the sum of two proper $\group{B}$-stable ideals can never be $\mathbb{C}[\mathfrak{m}^{\bullet}]$.
So the intersection of two non empty closed $\group{B}$-stable subsets of $\mathfrak{m}^{\bullet}$ is 
non empty.
\end{proof}

\begin{corollary}\label{main-corollary}
There is a unique closed $\group{B}$-orbit in $\mathfrak{m}^{\bullet}$.
\end{corollary}
\begin{proof}
Any $\group{B}$-orbit of minimal dimension is closed, so a closed $\group{B}$-orbit exists and
Theorem \ref{intersection-theorem} says that it is unique. 
\end{proof}

\section{Description of the unique closed orbit}\label{sec:description}

We shall show in this section that the unique closed $\group{B}$-orbit in Corollary \ref{main-corollary}
is the orbit of any element in $\mathfrak{m}$
whose support is $\soc{\mathfrak{m}}$. Our approach is to use the
following geometrical property concerning antichains of positive roots to study 
the action of the torus $\group{T}$ on the projective variety $\mathbb{P}(\mathfrak{m})$.

\begin{theorem}\label{antichain-hyperplane}
Let $\Gamma \subset \Delta^{+}$ be a non empty antichain. 
There exists $(H,n) \in \mathfrak{h}\times \mathbb{N}^{*}$ such that 
$\alpha (H) \in \mathbb{N}^{*}$ for all $\alpha \in \Pi$ and 
$\gamma  (H) = n$ 
for all $\gamma \in \Gamma$.
\end{theorem}
\begin{proof}
%We shall fix some notations.
Let $\ell = \sharp\Pi$ be the rank of $\Delta$, $\Pi = \{ \alpha_{1},\dots ,\alpha_{\ell}\}$,
$(H_{1},\dots ,H_{\ell})$ the basis of $\mathfrak{h}$ whose dual basis is $\Pi$.
Given $\alpha = \sum_{i=1}^{\ell} c_{i}\alpha_{i}\in\Delta$, we set 
$\Pi_{\alpha} = \{ \alpha_{i} \in \Pi \, ; c_{i}\neq 0 \} = \{\alpha_{i} \in \Pi  \, ; \alpha (H_{i}) \neq 0 \}$.
\smallskip

We shall prove the theorem by induction on $\ell$.
If $\ell = 1$, then $\Gamma = \Pi$. So 
$( H_{1}, 1)$ verifies the required conditions. 
Let us suppose that $\ell > 1$.

\smallskip
\noindent 1) Suppose that $\sharp \Gamma \leqslant 2$.
\smallskip

If $\Gamma = \{\gamma\}$, then
$(H_{1}+\cdots + H_{\ell}, \height( \gamma ) )$ verifies the required conditions.

If $\Gamma = \{\gamma ,\gamma'\}$ where
$\gamma = \sum_{i=1}^{\ell} c_{i}\alpha_{i}$ and $\gamma' = \sum_{i=1}^{\ell} c_{i}'\alpha_{i}$
are distinct, set 
$$
I_{+} = \{ i \, ; c_{i} -c_{i}' > 0 \} \ , \
I_{0} = \{ i \, ; c_{i}  = c_{i}' \} \ , \
I_{-} = \{ i \, ; c_{i} -c_{i}' < 0 \} \ .
$$
Since $\gamma$ and $\gamma'$ are non comparable, both $I_{+}$ and $I_{-}$ are non empty. 
Let $c_{\pm} = \sum_{i\in I_{\pm}} c_{i}$,
$c_{\pm}' = \sum_{i\in I_{\pm}} c_{i}'$, 
$c = \sum_{i\in I_{0}} c_{i} = \sum_{i\in I_{0}} c_{i}'$ and
$$\begin{array}{c}
H = \sum\limits_{i\in I_{+}} (c_{-}' -c_{-}) H_{i} +  \sum\limits_{i\in I_{-}} (c_{+}- c_{+}') 
H_{i} + \sum\limits_{i\in I_{0}} H_{i}.
\end{array}$$
Since $c_{-}' > c_{-} > 0$ and $c_{+} > c_{+}' > 0$, 
$(H , c_{-}'c_{+} - c_{-} c_{+}'+ c )$ verifies the required conditions. 

\smallskip
\noindent 2) Suppose that $\Pi_{\Gamma} = \bigcup_{\gamma\in \Gamma} \Pi_{\gamma} \neq \Pi$.
\smallskip

Let $\mathfrak{h}_{\Gamma}
= \mathrm{Vect}(H_{i} , \alpha_{i}\in\Pi_{\Gamma})$. Being a subset
of $\Pi$, $\Pi_{\Gamma}$ is an antichain. 
If $\Pi_{\Gamma} \neq \Pi$, then
$\rank \Delta_{\Pi_{\Gamma}} < \rank \Delta$. By induction, there exists 
$(H',n) \in \mathfrak{h}_{\Gamma} \times \mathbb{N}^{*}$ verifying the required conditions
for $\Gamma$ as an antichain in $\Delta_{\Pi_{\Gamma}}^{+}$. The pair 
$( H' + \sum_{\alpha_{i}\in \Pi \setminus \Pi_{\Gamma}} H_{i} , n)$
verifies the required conditions.

\smallskip
\noindent 3) Suppose that $\Pi_{\Gamma} = \Pi$ and $\Delta$ is not irreducible.
\smallskip

There is a partition 
$\Pi_{1} \cup \cdots \cup \Pi_{r}$ of $\Pi$ 
such that $\Delta = \Delta_{\Pi_{1}}\cup \cdots \cup \Delta_{\Pi_{r}}$ is the disjoint union of irreductible 
root systems of strictly lower rank. As $\Pi_{\Gamma} = \Pi$, 
$\Gamma = ( \Gamma\cap \Delta_{\Pi_{1}} ) \cup \cdots \cup
( \Gamma\cap \Delta_{\Pi_{r}})$ is the disjoint union of non empty antichains.
By induction, for $1\leqslant i \leqslant r$, there exists 
$(K_{i} , n_{i}) \in \mathfrak{h}_{\Pi_{i}} \times \mathbb{N}^{*}$ 
verifying the required conditions for $\Gamma\cap \Delta_{\Pi_{i}}$ as an antichain in 
$\Delta_{\Pi_{i}}^{+}$.  Then $( \sum_{i=1}^{r} \frac{n_{1}\cdots n_{r}}{n_{i}} K_{i} , n_{1}\cdots n_{r})$ 
verifies the required conditions.

\smallskip
\noindent 4) We are left with situation where $\sharp \Gamma\geqslant 3$, $\Pi_{\Gamma} = \Pi$ and
$\Delta$ is irreducible.
\smallskip

We shall use the numbering of simple roots of irreducible root systems in \cite[Chapter 18]{TY}. 
Let $\Delta_{1}^{+} = \{\alpha \in\Delta^{+} ; \alpha_{1}\in \Pi_{\alpha} \}$,
$\Gamma_{1} = \Gamma \cap \Delta_{1}^{+}$ and $\Gamma' = \Gamma \setminus \Gamma_{1}$.
%Note that $\Gamma'$ is an antichain in $\Delta_{\Pi (m)}$. 
Since $\Pi_{\Gamma} = \Pi$, the set
$\Gamma_{1}$ is non empty. Let us fix an element $\gamma \in \Gamma_{1}$.

The main observation here is 
that when $\Gamma'$ is non empty, the integer $m = \min \{  k \, ; \alpha_{k} \in \Pi_{\Gamma'} \} >1$, 
and as $\Pi_{\Gamma'} \subset  \Pi (m) = \{ \alpha_{m},\dots ,\alpha_{\ell}\}$, we may apply 
the induction hypothesis on the non empty antichain $\Gamma'$ in $\Delta_{\Pi (m)}$.
\smallskip

\noindent$\triangleright$ Case 1 : $\Delta$ is of type $A_{\ell}$, $B_{\ell}$ or $C_{\ell}$.
\smallskip

We have the following properties for these irreducible root systems.

\begin{itemize}
\item[(P1)] For $1 \leqslant i < j \leqslant \ell$, let 
$\alpha_{i,j} = \alpha_{i} + \alpha_{i+1} + \cdots + \alpha_{j}$. 
Non simple positive roots of $\Delta$ are of the form :

Type $A_{\ell}$ : $\alpha_{i,j}$, where $1 \leqslant i < j \leqslant \ell$.

Type $B_{\ell}$ : $\alpha_{i,j}$ or $\alpha_{i,\ell} + \alpha_{j,\ell}$, where 
$1 \leqslant i < j \leqslant \ell$.

Type $C_{\ell}$ : $\alpha_{i,j}$ or $\alpha_{i,\ell} + \alpha_{j-1,\ell-1}$,
where $1 \leqslant i < j \leqslant \ell$.
%\smallskip

\noindent In particular, for $\alpha\in\Delta^{+}$, $\Pi_{\alpha}$ consists of simple roots with 
consecutive indices, and we have either
$\alpha = \sum_{\beta\in \Pi_{\alpha}} \beta$ or 
$\alpha_{\ell} \in \Pi_{\alpha}$.
\end{itemize}

We obtain by (P1) that $\Delta_{1}^{+}$ is totally ordered with respect to $\prec$ in these cases.
So $\Gamma_{1} = \{ \gamma \}$. 
By induction, there exists $( H' , n' ) \in\mathfrak{h}_{\Pi (m)} \times \mathbb{N}^{*}$ 
verifying the required conditions for $\Gamma'$ as an antichain in $\Delta_{\Pi (m) }$.

Let us fix $\gamma'\in \Gamma'$ such that $\alpha_{m}\in\Pi_{\gamma'}$, and
write $\gamma = \beta + \delta$ where 
$\beta = \sum_{i=1}^{m-1} c_{i}\alpha_{i}$ and 
$\delta = \sum_{i=m}^{\ell} c_{i}\alpha_{i}$.
Since $\gamma$ and $\gamma'$ are non comparable, we check using (P1)
that $\gamma' - \delta$ is a non zero sum of simple roots in $\Pi (m)$. So
$0 \leqslant  \delta (H')  < \gamma'  (H') = n'$. 
Set 
$$\begin{array}{c}
H = \sum\limits_{i=1}^{m-1} (n' - \delta( H') ) H_{i} + \height_{\Pi} (\beta) H'.
\end{array}$$
Since $\alpha_{1}\in \Pi_{\gamma_{1}}$, we have $\height_{\Pi} (\beta ) \in \mathbb{N}^{*}$,
and therefore $(H , \height_{\Pi}(\beta ) n' )$ verifies the required conditions.
\smallskip

\noindent$\triangleright$ Case 2 : $\Delta$ is of type $D_{\ell}$, $\ell \geqslant 4$.
\smallskip

For $1\leqslant i < j \leqslant \ell-2$, let 
$\alpha_{i,j} = \alpha_{i} + \alpha_{i+1} + \cdots + \alpha_{j}$.

\begin{itemize}
\item[(P2)] 
Non simple positive roots in $\Delta$ are of the form 
$$
\alpha_{i,j}, \alpha_{k,\ell-2} + \alpha_{\ell-1},  \alpha_{k,\ell-2} + \alpha_{\ell},
\alpha_{k,\ell-2} + \alpha_{\ell-1} + \alpha_{\ell}, \alpha_{i,\ell-2} + \alpha_{j,\ell -2} + 
\alpha_{\ell-1} + \alpha_{\ell} 
$$
\end{itemize}
\noindent where $1\leqslant i < j \leqslant \ell-2$ and $1\leqslant k \leqslant \ell-2$. In particular, if
$\alpha = \sum_{i=1}^{\ell} c_{i} \alpha_{i}$ is a positive root, then 
$c_{i} \in \{ 0, 1,2 \}$ for all $i$, and if $c_{i} = 2$, then $2\leqslant i\leqslant \ell-2$ and
$c_{i-1} = c_{\ell-1} = c_{\ell} = 1$. Also, if $c_{\ell-1} = c_{\ell} = 1$, then $c_{\ell-2}\neq 0$.
\smallskip

The set $\Delta_{1}^{+}$ is not totally ordered
with respect to $\prec$. However, only $\alpha_{1,\ell-2} + \alpha_{\ell-1}$ 
and $\alpha_{1,\ell-2} + \alpha_{\ell}$ are non comparable in $\Delta_{1}^{+}$. So 
$\sharp \Gamma_{1} \leqslant 2$. This implies that $\Gamma'$ is non empty
because $\sharp \Gamma\geqslant 3$.

We shall adapt the arguments in Case 1 according to $\sharp\Gamma_{1}$.
\smallskip

\noindent $\circ$ Subcase (i) : $\sharp \Gamma_{1} = 1$. 
\smallskip

Let $\gamma'$, $\beta$, $\delta$ be as in Case 1.
If $\sharp (\Pi_{\gamma}\cap \{ \alpha_{\ell-1},\alpha_{\ell}\} ) \neq 1$, then by (P2),
$\gamma'-\delta$ is a non zero sum of simple roots of $\Pi (m)$. So we may apply 
the same arguments as in Case 1 to obtain $(H,n)$ verifying the required conditions.

Suppose that $\sharp (\Pi_{\gamma}\cap \{ \alpha_{\ell-1},\alpha_{\ell}\} ) = 1$.
Then either $\gamma = \alpha_{1,\ell-2}+\alpha_{\ell-1}$ or 
$\gamma = \alpha_{1,\ell-2}+\alpha_{\ell}$. By symmetry of the Dynkin diagram of $\Delta$,  
we may assume that $\gamma = \alpha_{1,\ell-2}+\alpha_{\ell-1}$. 

Since $\gamma$ and $\gamma'$ are non comparable, (P2) implies that 
$\alpha_{\ell} \in \Pi_{\gamma'}$. If $\alpha_{\ell-1} \in \Pi_{\gamma'}$, then
again by (P2), $\gamma' - \delta$ is a non zero sum of simple roots in $\Pi (m)$. So
we may apply the same arguments as in Case 1 to obtain $(H,n)$ 
verifying the required conditions.

Suppose now that $\alpha_{\ell-1} \not\in \Pi_{\gamma'}$. This implies 
that $\gamma'$ is either $\alpha_{\ell}$ or $\alpha_{m,\ell-2} + \alpha_{\ell}$.
Since $\sharp \Gamma \geqslant 3$, the set $\Gamma'' = \Gamma \setminus 
\{ \gamma , \gamma' \}$ is non empty. By (P2) and the fact that $\Gamma$ is an antichain,
for any $\theta\in \Gamma''$, we have $\alpha_{\ell-1},\alpha_{\ell} \in  
\Pi_{\theta} \subset \{ \alpha_{m+1},\dots , \alpha_{\ell}\}$. Let $m'$ be minimal such that 
$\alpha_{m'} \in\Pi_{\Gamma''}$. Then $m < m' \leqslant \ell-2$. Fix 
$\gamma'' \in \Gamma''$ such that $\alpha_{m'} \in \Pi_{\gamma''}$.

By induction, there exist $H'' = \sum_{i=m'}^{\ell} b_{i}H_{i} \in \mathfrak{h}_{\Pi(m')}$
and $n'' \in \mathbb{N}^{*}$ such that $(H'',n'')$ verifies the required conditions for $\Gamma''$
as an antichain in $\Delta_{\Pi(m')}$. Since 
$\alpha_{\ell} \in  \Pi_{\theta}$ for any $\theta \in \Gamma''$, $(H'' + k H_{\ell} , n'' + k )$ verifies also
the required conditions for any $k\in \mathbb{N}$. Thus
we may assume that $b_{\ell} > b_{\ell-1}$.  

We have $\gamma' = \beta' + \delta'$ where $\beta' = \alpha_{m,m'-1}$ and 
$\delta' = \alpha_{m',\ell-2}+\alpha_{\ell}$. By (P2), $\gamma'' - \delta'$ is a non zero sum 
of simple roots in $\Pi (m')$. It follows that 
$0 \leqslant \delta' ( H'') < \gamma'' (H'') = n''$. Set
$$\begin{array}{c}
H' = \sum\limits_{i=m}^{m'-1} ( n'' - \delta' (H'') ) H_{i}
+ (m'-m) H''
\end{array}$$
Since $m' > m$, $(H', (m'-m) n'')$ verifies the required conditions for $\Gamma'$ as an antichain in 
$\Delta(m)$.

Now
$
0 \leqslant \delta (H') = (m'-m) (n'' -\delta'(H'') + \delta (H'')  )
= (m'-m) (n'' - b_{\ell} + b_{\ell-1})$.
Thus $\delta (H') < (m'-m) n''$
because $b_{\ell} > b_{\ell-1}$. 
Consequently, by setting 
$$\begin{array}{c}
H = \sum\limits_{i=1}^{m-1} \bigl( (m'-m) n'' -  \delta ( H' )  \bigr) H_{i} + (m-1) H', 
\end{array}$$
the pair $\bigl( H , (m-1)(m'-m)n'' \bigr)$ verifies the required conditions.

\smallskip
\noindent $\circ$ Subcase (ii) : $\sharp \Gamma_{1} = 2$.
\smallskip

In this case, 
$\Gamma_{1} = \{ \alpha_{1,\ell-2} + \alpha_{\ell-1} , \alpha_{1,\ell-2} + \alpha_{\ell} \}$.
The fact that $\Gamma$ is an antichain implies that
$\alpha_{\ell-1}, \alpha_{\ell}\in \Pi_{\theta}$ for any $\theta \in \Gamma'$. 
Let us fix $\gamma' \in \Gamma'$ such that $\alpha_{m}\in\Pi_{\gamma'}$.

By induction, there exist $H' = \sum_{i=m}^{\ell} b_{i}H_{i} \in \mathfrak{h}_{\Pi(m)}$
and $n' \in \mathbb{N}^{*}$ such that $(H',n')$
verifies the required conditions for $\Gamma'$ as an antichain in $\Delta_{\Pi (m)}$.
Set 
$$\begin{array}{c}
H'' = \sum\limits_{i=m}^{\ell-2} 2 b_{i}H_{i}  + (b_{\ell-1} + b_{\ell} ) ( H_{\ell-1} + H_{\ell} ).
\end{array}$$
Since $\alpha_{\ell-1}, \alpha_{\ell} \in \Pi_{\beta}$  for any 
$\beta \in \Gamma'$, it follows from (P2) that $(H'' , 2n')$ verifies also the required conditions
for $\Gamma'$ as an antichain in $\Delta_{\Pi (m)}$.

We have $\alpha_{1,\ell-2} + \alpha_{\ell-1} = \alpha_{1,m-1} + \delta_{1}$ and
$\alpha_{1,\ell-2} + \alpha_{\ell}  = \alpha_{1,m-1} + \delta_{2}$ where 
$\delta_{1} = \alpha_{m,\ell-2} +\alpha_{\ell-1}$ and $\delta_{2} = \alpha_{m,\ell-2} +\alpha_{\ell}$.
As in Subcase (i), we check readily that $\gamma' - \delta_{1}$ and $\gamma'-\delta_{2}$ 
are non zero sums of simple roots in $\Pi (m)$. It follows that
$0 \leqslant \delta_{1} (H'') = \delta_{2} (H'' )  < \gamma'(H'') = 2n'$.
Set 
$$\begin{array}{c}
H = \sum\limits_{i=1}^{m-1} (2n' - \delta_{1}(H'') ) H_{i} + 
(m-1) H''.
\end{array}$$
Then $( H ,  2(m-1)n' )$ verifies the required conditions.
\smallskip

\noindent$\triangleright$ Case 3 : $\Delta$ is of exceptional type.
\smallskip

The case $\Delta$ is of type $G_{2}$ is void since $\sharp \Gamma \geqslant 3$.
For the other exceptional types, $\sharp \Delta_{1}^{+}$ is 
$16, 33, 78, 15$ respectively for $E_{6}$, $E_{7}$, $E_{8}$ and $F_{4}$, the maximal 
cardinality of $\Gamma_{1}$ is $2, 3, 5, 2$ and the number of possibilities for $\Gamma_{1}$ is
$26, 119, 1348, 22$. So there are too many cases to consider for $\Gamma_{1}$
to try to adapt the arguments in case 1. 

Observe that if $\Gamma$ contains a simple root $\alpha_{i}$,
then $\Pi_{\Gamma \setminus \{ \alpha_{i}\} } = \Pi \setminus \{\alpha_{i}\}$ because
$\Pi_{\Gamma} = \Pi$. By induction, 
there exists $(H',n') \in \mathfrak{h}_{\Pi \setminus \{ \alpha_{i}\} } \times \mathbb{N}^{*}$
verifying the required conditions for $\Gamma \setminus \{ \alpha_{i}\}$ as an antichain in 
$\Delta_{\Pi \setminus \{ \alpha_{i}\} }$. So $(H' + n' H_{i} , n' )$ verifies the required
conditions. Thus we are left to prove the result for $\Gamma$ 
verifying $\sharp \Gamma \geqslant 3$, $\Pi_{\Gamma} =\Pi$, 
$\Gamma\cap \Pi = \emptyset$ and $\Gamma$ is maximal 
by inclusion. The number of such $\Gamma$  is 
$91$ for $E_{6}$, $512$ for $E_{7}$, $3289$ for $E_{8}$ and $10$ for $F_{4}$.

Computations were carried out using {\sc Gap4} to obtain in all these cases $(H,n)$ verifying the required 
conditions. The table below gives $(H,n)$ for the $10$ cases in $F_{4}$. 

{\scriptsize$$
\begin{array}{c|c}
\Gamma & ( H , n ) \\ 
\hline \hline
\{ \alpha_{1}+\alpha_{2} , \alpha_{2}+\alpha_{3} , \alpha_{3}+\alpha_{4} \}
& ( H_{1}+H_{2}+H_{3}+H_{4} , 2 ) \\
\hline 
\{ \alpha_{1}+\alpha_{2} , \alpha_{2}+2\alpha_{3} ,\alpha_{3}+\alpha_{4} \}
& (2H_{1} + H_{2} + H_{3}+2H_{4} , 3)  \\ 
\hline 
\{ \alpha_{1}+\alpha_{2} , \alpha_{2}+2\alpha_{3} ,\alpha_{2}+\alpha_{3}+\alpha_{4} \}
& (2H_{1} + H_{2} + H_{3}+H_{4} , 3) \\
\hline 
\{ \alpha_{1}+\alpha_{2}+\alpha_{3} , \alpha_{2}+2\alpha_{3} , \alpha_{3}+\alpha_{4} \}
& (H_{1} + H_{2}+ H_{3}+2H_{4} , 3) \\
\hline 
\{ \alpha_{1}+\alpha_{2}+\alpha_{3} , \alpha_{2}+2\alpha_{3} ,\alpha_{2}+\alpha_{3}+\alpha_{4} \} 
& (H_{1}+H_{2}+H_{3}+H_{4} , 3) \\
\hline 
\{ \alpha_{1}+\alpha_{2}+2\alpha_{3} , \alpha_{1}+\alpha_{2}+\alpha_{3}+\alpha_{4} ,
\alpha_{2}+2\alpha_{3}+\alpha_{4} \} & (H_{1}+H_{2}+H_{3}+H_{4} , 4) \\ 
\hline 
\{ \alpha_{1}+\alpha_{2}+2\alpha_{3} , \alpha_{1}+\alpha_{2}+\alpha_{3}+\alpha_{4} ,
\alpha_{2}+2\alpha_{3}+2\alpha_{4} \} & ( 2H_{1} + H_{2}+H_{3}+H_{4}, 5) \\
\hline 
\{ \alpha_{1}+2\alpha_{2}+2\alpha_{3} , \alpha_{1}+\alpha_{2}+\alpha_{3}+\alpha_{4} ,
\alpha_{2}+2\alpha_{3}+\alpha_{4} \} & (H_{1}+H_{2}+H_{3}+2H_{4} , 5) \\ 
\hline 
\{ \alpha_{1}+2\alpha_{2}+2\alpha_{3} , \alpha_{1}+\alpha_{2}+\alpha_{3}+\alpha_{4} ,
\alpha_{2}+2\alpha_{3}+2\alpha_{4} \} & (3H_{1} +H_{2}+H_{3}+2H_{4} , 7) \\ 
\hline 
\{ \alpha_{1}+2\alpha_{2}+2\alpha_{3} , \alpha_{1}+\alpha_{2}+2\alpha_{3}+\alpha_{4} ,
\alpha_{2}+2\alpha_{3}+2\alpha_{4} \} & (H_{1}+H_{2}+H_{3}+H_{4} , 5) 
\end{array}
$$}

\noindent Note that the values $n$ obtained in our computations are minimal. For the other types, 
the largest $n$ found is $9$ for $E_{6}$, $14$ for $E_{7}$ and $25$ for $E_{8}$.
\end{proof}

\begin{remarks}
\begin{enumerate}
\item In the proof of Theorem \ref{antichain-hyperplane}, we have $\Pi_{\Gamma'} =\Pi (m)$ in Case 1 
and most cases in Case 2.
\item Our proof relies on a type by type analysis of positive roots. As this is a geometrical property,
it would be nice to have a more geometric proof.
\end{enumerate}
\end{remarks}

Observe that the natural $\mathbb{C}^{*}$-action on $\mathfrak{n}$
commutes with the action of $\group{B}$. We have therefore an action of 
$\widehat{\group{B}} = \group{B} \times \mathbb{C}^{*}$ 
(resp. $\widehat{\group{T}} = \group{T} \times \mathbb{C}^{*}$) 
on $\mathfrak{n}$ given by 
$(\sigma, \lambda ) (X) = \sigma (\lambda X) = \lambda \sigma (X)$
for $(\sigma, \lambda )\in \widehat{\group{B}}$ and $X\in\mathfrak{n}$.
For any group $\group{H}$ acting on $\mathfrak{n}$ and $X\in\mathfrak{n}$,
we shall denote by $\Omega_{\group{H}}(X)$ the $\group{H}$-orbit of $X$.

\begin{corollary}\label{description}
Let $\mathfrak{m}$ be a nilpotent ideal of $\mathfrak{b}$ and 
$X\in \mathfrak{m}$ be such that $\supp (X) = \soc{\mathfrak{m}}$.
\begin{enumerate}
\item The $\group{B}$-orbit $\Omega_{\group{B}}(X)$ is conical.
\item For any $Y\in\mathfrak{m}^{\bullet}$, we have
$X \in \overline{\Omega_{\widehat{\group{T}}} (Y)}$.
\item $\Omega_{\group{B}}(X)$ is the unique closed $\group{B}$-orbit of $\mathfrak{m}^{\bullet}$.
\end{enumerate}
\end{corollary}
\begin{proof}
1. By Theorem \ref{antichain}, the set $\soc{\mathfrak{m}}$ is linearly independent.
Therefore 
\begin{equation}\label{eq:conical}
\Omega_{\group{T}}(X) = \{ Y \in \mathfrak{m}\, ; \supp (Y) = \soc{\mathfrak{m}}\}.
\end{equation}
Hence $\Omega_{\mathbb{C}^{*}}( X )\subset \Omega_{\group{T}}(X)$, which in turn implies that 
$\Omega_{\widehat{\group{B}}} (X) = \Omega_{\group{B}}(X)$. 

2. By Theorem \ref{antichain-hyperplane}, $\soc{\mathfrak{m}}$ is exactly the set of elements in
$\supp (Y)$ lying on a face of the convex hull
of $\supp (Y)$. We obtain from \cite[Proposition 2.18]{C} that 
$\overline{\Omega_{\widehat{\group{T}}} (Y)}$ contains
an element in $\mathfrak{m}$ whose support is $\soc{\mathfrak{m}}$. 
So the result follows by \eqref{eq:conical}.

3. By Corollary \ref{main-corollary}, it suffices to show that $\Omega_{\group{B}} (X)$ is closed in 
$\mathfrak{m}^{\bullet}$. Now points 1 and 2 say that for any $Y\in\mathfrak{m}^{\bullet}$, we have
$
\Omega_{\group{B}}(X)  = \Omega_{\widehat{\group{B}}} (X) 
\subset  \overline{\Omega_{\widehat{\group{B}}} (Y)}.
$
It follows that if $Y\not\in \Omega_{\group{B}} (X)$, then 
$$
\dim \Omega_{\group{B}}(X) = \dim \Omega_{\widehat{\group{B}}} (X)
< \dim  \overline{\Omega_{\widehat{\group{B}}} (Y)} \leqslant \dim \Omega_{\group{B}}(Y) + 1.
$$
Thus $\Omega_{\group{B}}(X)$ is a $\group{B}$-orbit of minimal dimension in $\mathfrak{m}^{\bullet}$,
so it is closed.
\end{proof}

\begin{remarks}\label{conical}
\begin{enumerate}
\item It follows immediately from Corollary \ref{description} that an element
$X\in\mathfrak{b}$ is nilpotent if and only if $0 \in \overline{\Omega_{\group{B}} ( X)}$.
\item $\group{B}$-orbits of nilpotent elements in $\mathfrak{b}$ are not in general conical. 
Here are examples of non conical $\group{B}$-orbits which are minimal with respect to the 
rank of a given type where we use the numbering of simple roots in \cite{TY}. 
\begin{itemize}
\item[$A_{6}$ :] $X_{\alpha_{1}}+X_{\alpha_{3}}+X_{\alpha_{4}}+X_{\alpha_{6}}+
X_{\alpha_{1}+\alpha_{2}}+X_{\alpha_{5}+\alpha_{6}}+
X_{\alpha_{2}+\alpha_{3}+\alpha_{4}+\alpha_{5}}$
\item[$B_{4}$ :] $X_{\alpha_{2}}+X_{\alpha_{4}}+X_{\alpha_{1}+\alpha_{2}}+X_{\alpha_{2}+\alpha_{3}}
+X_{\alpha_{1}+\alpha_{2}+2\alpha_{3}+2\alpha_{4}}$
\item[$C_{4}$ :] $X_{\alpha_{1}}+X_{\alpha_{3}}+X_{\alpha_{4}}+X_{\alpha_{1}+\alpha_{2}}+
X_{2\alpha_{2}+2\alpha_{3}+\alpha_{4}}$
\item[$D_{5}$ :] $X_{\alpha_{2}}+X_{\alpha_{4}}+X_{\alpha_{5}}+X_{\alpha_{1}+\alpha_{2}}+
X_{\alpha_{2}+\alpha_{3}}+X_{\alpha_{1}+\alpha_{2}+2\alpha_{3}+\alpha_{4}+\alpha_{5}}$
\item[$E_{6}$ :] $X_{\alpha_{1}}+X_{\alpha_{2}}+X_{\alpha_{3}}+X_{\alpha_{5}}
+X_{\alpha_{6}} +X_{\alpha_{2}+\alpha_{4}}+ X_{\alpha_{1}
+\alpha_{2}+2\alpha_{3}+3\alpha_{4}+2\alpha_{5}+\alpha_{6}}$
\item[$F_{4}$ :] $X_{\alpha_{1}}+X_{\alpha_{3}}+X_{\alpha_{4}}+X_{\alpha_{1}+\alpha_{2}}
+X_{\alpha_{1}+3\alpha_{2}+4\alpha_{3}+2\alpha_{4}}$
\end{itemize}
\end{enumerate}
\end{remarks}

\end{document}